\documentclass[preprint, 10pt]{amsart}

\usepackage{graphicx,subfigure}
\usepackage{epstopdf}
\usepackage{amsmath,amsthm,amssymb}
\usepackage{rotating}
\usepackage{morefloats}
\usepackage{comment}

\newtheorem{thm}{Theorem}

\newtheorem{lem}{Lemma}

\newtheorem{exmp}{Example}

\begin{document}
\title{Geometric shrinkage priors for K\"ahlerian signal filters}
\author{Jaehyung Choi}
\address{Department of Applied Mathematics and Statistics\\
 SUNY, Stony Brook, NY 11794, USA}
\email{jj.jaehyung.choi@gmail.com}
\author{Andrew P. Mullhaupt}
\address{Department of Applied Mathematics and Statistics\\
 SUNY, Stony Brook, NY 11794, USA}
\email{doc@zen-pharaohs.com}

\begin{abstract}
	We construct geometric shrinkage priors for K\"ahlerian signal filters. Based on the characteristics of K\"ahler manifolds, an efficient and robust algorithm for finding superharmonic priors which outperform the Jeffreys prior is introduced. Several ans\"atze for the Bayesian predictive priors are also suggested. In particular, the ans\"atze related to K\"ahler potential are geometrically intrinsic priors to the information manifold of which the geometry is derived from the potential. The implication of the algorithm to time series models is also provided.
\end{abstract}

\maketitle
\section{Introduction}

	In information geometry, signal processing is one of the most important applications. In particular, an information geometric approach to various linear time series models has been also well-known~\cite{Amari:2000,Ravishanker:1990p5895,Ravishanker:2001p5836,Barbaresco:2006, Tanaka:2006, Tanaka:2009, Choi:2014c}. The geometric description of the linear systems is not confined to the pursuit of mathematical beauty.  Komaki's work \cite{Komaki:2006} is in the line of developing practical tools for Bayesian inference. Using the Kullback--Leibler divergence as a risk function for estimation, he found that superharmonic shrinkage priors outperform the Jeffreys prior in the viewpoint of information theory. Better prediction in the Bayesian framework is attainable by the Komaki priors.
	
	However, a difficult part of Komaki's idea in practice is verifying whether or not a prior function is superharmonic. In particular, when high-dimensional statistical manifolds are considered, it is technically tricky to test the superharmonicity of prior functions because Laplace--Beltrami operators on the manifolds are non-trivial. Although some superharmonic priors for the autoregressive (AR) models were found not only in the two-dimensional cases \cite{Tanaka:2006, Choi:2014c} but also in arbitrary dimensions \cite{Tanaka:2009}, there is no clue about the Bayesian shrinkage priors of more complicated models such as the autoregressive moving average (ARMA) models, the fractionally integrated ARMA (ARFIMA) models, and any arbitrary signal filters. Additionally, generic algorithms for systematically obtaining the information shrinkage priors are not known yet.
	
	The connection between K\"ahler manifolds and information geometry has been reported \cite{Barndorff-Nielsen:1997,Barbaresco:2006,Barbaresco:2012,Zhang:2013,Barbaresco:2014} and the mathematical correspondence between a K\"ahler manifold and the information geometry of a linear system is recently revealed. It is found that the information geometry of a signal filter with a finite complex cepstrum norm is a K\"ahler manifold \cite{Choi:2014c}. In particular, the Hermitian condition on the K\"ahlerian information manifolds is clearly seen under conditions on the transfer function of the linear system. Moreover, many practical aspects of introducing K\"ahler manifolds to information geometry for signal processing were also reported in the same literature \cite{Choi:2014c}. One of the benefits in the K\"ahlerian information geometry is that the simpler form of the Laplace--Beltrami operator on the K\"ahler manifold is beneficial to finding the Komaki priors. 	
	
	In this paper, we construct Komaki-style shrinkage priors for K\"ahlerian signal filters. By introducing an algorithm which is based on the characteristics of K\"ahler manifolds, the Bayesian predictive priors outperforming the Jeffreys prior can be obtained in a more efficient and more robust way. Several prior ans\"atze are also suggested. Among the ans\"atze, the geometric shrinkage priors related to K\"ahler potential are intrinsic priors on the information manifold because the geometry is given by the K\"ahler potential. We also provide the geometric priors for the ARFIMA models where the Komaki priors have not been reported. The structure of this paper is as follows. In next section, theoretical backgrounds of K\"ahlerian information geometry and superharmonic priors are introduced. In Section \ref{sec_geo_prior_algo_prior}, an algorithm and ans\"atze for the geometric shrinkage priors are suggested.  The implication of the algorithm to the ARFIMA models is given in Section \ref{sec_geo_prior_example}. We conclude the paper in the last section.
	
\section{Theoretical Backgrounds}
\label{sec_geo_prior_back}
\subsection{K\"ahlerian Filters}

	A linear filter with $n$-dimensional complex parameters $\boldsymbol{\xi}$ is characterized by a transfer function $h(w;\boldsymbol{\xi})$ in the frequency domain $w$ with 
	\begin{equation}
		y(w)=h(w;\boldsymbol{\xi})x(w)\nonumber
	\end{equation}
	where $y$ and $x$ are complex output and input signals, respectively. A spectral density function $S(w;\boldsymbol{\xi})$ is defined as the absolute square of the transfer function
	\begin{equation}
		S(w;\boldsymbol{\xi})=|h(w;\boldsymbol{\xi})|^2\nonumber
	\end{equation}
	and it is a real-valued measurable quantity. 

	In information geometry, it is well-known by Amari and Nagaoka \cite{Amari:2000} that the geometry of a linear system is determined by the spectral density function $S(w;\boldsymbol{\xi})$ under the stability condition, minimum phase, and 
\begin{equation}
		\frac{1}{2\pi}\int_{-\pi}^{\pi} |\log{S(w;\boldsymbol{\xi})}|^2 dw <\infty.\nonumber
	\end{equation}
	The last condition is also known as the finite unweighted norm of the power cepstrum of a filter \cite{Bogert:1967,Martin:2000}. For a linear system with the spectral density function satisfying the above conditions, the metric tensor of the information geometry is given by	
\begin{equation}
	\label{metric_sdf}
	g_{\mu\nu}(\boldsymbol{\xi})=\frac{1}{2\pi}\int_{-\pi}^{\pi} (\partial_\mu \log{S}) (\partial_\nu \log{S}) dw\nonumber
	\end{equation}
	where the partial derivatives are taken with respect to the model parameters $\boldsymbol{\xi}$. 
	
	The metric tensor can be expressed in a complexified coordinate system and the Z-transformed transfer function. With the Z-transformation, the holomorphic transfer function can be written in the form of series expansion of $z$
	\begin{equation}
		\label{trans_func_z}
		h(z;\boldsymbol{\xi})=\sum_{r=0}^{\infty}h_r(\boldsymbol{\xi}) z^{-r}
	\end{equation}
	where $h_r$ is an impulse response function. The Z-transformed power spectrum is also defined in the similar way. In this case, the conditions on the transfer function for constructing information geometry are identical to the spectral density function representation except for
	\begin{equation}
	\label{transfer_cepstrum_norm}
		\frac{1}{2\pi i}\oint_{|z|=1} |\log{h(z;\boldsymbol{\xi})}|^2 \frac{dz}{z} <\infty\nonumber
	\end{equation}
	and it is a necessary condition for the finite power cepstrum norm. The condition indicates that the Hardy norm of the logarithmic transfer function, also known as the unweighted complex cepstrum norm~\cite{Oppenheim:1965, Martin:2000}, is finite. The metric tensor of the geometry is given by the transfer function,
	\begin{align}
		\label{metric_complex_coordinate1}
		g_{ij}(\boldsymbol{\xi})=\frac{1}{2\pi i}\oint_{|z|=1} \partial_i \log{h(z;\boldsymbol{\xi})} \partial_j \log{h(z;\boldsymbol{\xi})} \frac{dz}{z}\\
		\label{metric_complex_coordinate2}
		g_{i\bar{j}}(\boldsymbol{\xi})=\frac{1}{2\pi i}\oint_{|z|=1} \partial_i \log{h(z;\boldsymbol{\xi})} \partial_{\bar{j}} \log{\bar{h}(\bar{z};\bar{\boldsymbol{\xi}})} \frac{dz}{z}
	\end{align}
	where $i,j$ run from $1$ to $n$ and $g_{\bar{i}\bar{j}}, g_{\bar{i}j}$ are the complex conjugates of $g_{ij}$ and $g_{i\bar{j}}$, respectively.
	
	After plugging the Z-transformed transfer function, Equation  (\ref{trans_func_z}), into the metric tensor expressions, Equations (\ref{metric_complex_coordinate1}) and  (\ref{metric_complex_coordinate2}), the metric tensor is expressed with the series expansion coefficients in $z$ of the logarithmic transfer function by
	\begin{align}
%		\label{metric_tensor_gauge1}
		g_{ij}&=\partial_i \eta_0\partial_j \eta_0\nonumber\\
%		\label{metric_tensor_gauge2}
		g_{i\bar{j}}&=\partial_i \eta_0\partial_{\bar{j}}\bar{\eta}_0+\sum_{r=1}^{\infty} \partial_i \eta_r \partial_{\bar{j}} \bar{\eta}_r\nonumber
	\end{align}
	where $\eta_r$ is the coefficient of $z^{-r}$ in the series expansion of the logarithmic transfer function, also known as a complex cepstrum coefficient \cite{Oppenheim:1965}. It is obvious that $\eta_0=\log{h_0}$.
	
	Recently, it is found by Choi and Mullhaupt \cite{Choi:2014c} that the information geometry of a linear system with a finite Hardy norm of a logarithmic transfer function (or the complex cepstrum norm) is the K\"ahler manifold that is the Hermitian manifold with the closed K\"ahler two-form: $g_{ij}=g_{\bar{i}\bar{j}}=0$ for the Hermitian manifold and $\partial_i g_{j\bar{k}}=\partial_j g_{i\bar{k}}$, $\partial_{\bar{i}} g_{k\bar{j}}=\partial_{\bar{j}} g_{k\bar{i}}$ for the closed K\"ahler two-form. Additionally, the Hermitian structure can be explicitly seen in the metric tensor if and only if the impulse response function with the highest degree in $z$, {\em i.e.}, $h_0$ in the unilateral transfer function case, is a constant in model parameters $\boldsymbol{\xi}$. In this paper, for simplicity, we only consider unilateral transfer functions with non-zero $h_0$ and the K\"ahler manifolds with the explicit Hermitian conditions on the metric tensors because complex manifolds are always Hermitian manifolds \cite{Nakahara:2003}. In this case, the necessary and sufficient condition for being a K\"ahler manifold is that $h_0(\boldsymbol{\xi})$ is a constant in $\boldsymbol{\xi}$ \cite{Choi:2014c}. 
	
	According to Choi and Mullhaupt \cite{Choi:2014c}, the benefits of the K\"ahlerian description are the followings. First of all, geometric objects are straightforwardly computed on a K\"ahler manifold. The non-trivial metric tensor component is simply derived from the following formula
	\begin{equation}
		\label{eqn_metric_kahler_formula}
		g_{i\bar{j}}=\partial_i\partial_{\bar{j}}\mathcal{K}
	\end{equation}
	where $\mathcal{K}$ is the K\"ahler potential of the geometry. The K\"ahler potential in the information geometry of a linear filter is the square of the Hardy norm (or $H^2$-norm) of the logarithmic transfer function (or the square of the complex cepstrum norm) on the unit disk $\mathbb{D}$
	\begin{equation}
		\label{eqn_kahler_potential_formula}
		\mathcal{K}=\frac{1}{2\pi i}\oint_{|z|=1}|\log{h(z;\boldsymbol{\xi})}|^2\frac{dz}{z}=||\log{h(z;\boldsymbol{\xi})}||^2_{H^2}
	\end{equation}
	and the details of the derivation are given in the literature \cite{Choi:2014c}. The non-trivial components of the Levi--Civita connection are expressed as 
	\begin{equation}
		\label{eqn_connection_formula}
		\Gamma_{ij,\bar{k}}=\partial_ig_{j\bar{k}}=\partial_i\partial_j\partial_{\bar{k}} \mathcal{K}
	\end{equation}
	and the other connection components are all vanishing. Notice that it is much simpler than the connection components on a non-K\"ahler manifold given by
	\begin{equation*}
		\Gamma_{ij,k}=\frac{1}{2}(\partial_ig_{jk}+\partial_jg_{ik}-\partial_kg_{ij})
	\end{equation*}
	and it is obvious that the number of calculation steps is significantly reduced in the K\"ahler case. The Riemann curvature tensor of the linear system geometry is also represented in the simpler form which is given in Choi and Mullhaupt \cite{Choi:2014c}. The Ricci tensor on the K\"ahler manifold is obtained as
	\begin{equation}
		\label{eqn_ricci_formula}
		R_{i\bar{j}}=-\partial_i\partial_{\bar{j}}\log{\mathcal{G}}
	\end{equation}
	where $\mathcal{G}$ is the determinant of the metric tensor. It is evident that we can skip the calculation of the Riemann curvature tensor in order to compute the Ricci tensor on a K\"ahler manifold.
	
	Additionally, the $\alpha$-generalization of the geometric objects is linear in $\alpha$ on K\"ahler manifolds. Since the Riemann curvature tensor on a K\"ahler manifold is linear in the $\alpha$-connection which is $\alpha$-linear, the Riemann tensor also exhibits the $\alpha$-linearity which leads to the $\alpha$-linear Ricci tensor and scalar curvature. 
	
	In addition to these advantages, any submanifolds of a K\"ahler manifold are also K\"ahler manifolds. If the information geometry of a given statistical model is a K\"ahler manifold, its submodels also have K\"ahler manifolds as the information geometry and all the properties of the ambient manifold are also equipped with the submanifolds. 
	
	Lastly, the K\"ahlerian information geometry is also useful to find superharmonic priors because of the simpler Laplace--Beltrami operators on the manifolds. We will cover the details of the superharmonic priors soon.
		
\subsection{Superharmonic Priors}

	For further discussions, we need to introduce the superharmonic priors suggested by Komaki \cite{Komaki:2006}. When we want to find the true probability distribution $p(y|\boldsymbol{\xi})$ based on given samples $x$ of size $N$, one of the best approaches is using Bayesian predictive density $p_\pi (y|x^{(N)})$ with a prior $\pi(\boldsymbol{\xi})$:
	\begin{equation}
		p_\pi (y|x^{(N)})=\frac{\int p(y|\boldsymbol{\xi}) p(x^{(N)}|\boldsymbol{\xi})\pi(\boldsymbol{\xi}) d\boldsymbol{\xi}}{\int p(x^{(N)}|\boldsymbol{\xi})\pi(\boldsymbol{\xi}) d\boldsymbol{\xi}}.\nonumber
	\end{equation}
	The superharmonic priors $\pi_I$ are derived from the difference between two risk functions with respect to the true probability density, one from the Jeffreys prior and another from the superharmonic prior:
	\begin{align}
		\mathbb{E}[D_{KL}(p(y|\boldsymbol{\xi})||p_{\pi_J}(y|x^{(N)}))|\boldsymbol{\xi}]-\mathbb{E}[D_{KL}(p(y|\boldsymbol{\xi})||p_{\pi_I}(y|x^{(N)}))|\boldsymbol{\xi}]\nonumber\\
		=\frac{1}{2N^2}g^{ij}\partial_i\log{\Big(\frac{\pi_I}{\pi_J}\Big)}\partial_j\log{\Big(\frac{\pi_I}{\pi_J}\Big)}-\frac{1}{N^2}\frac{\pi_J}{\pi_I}\Delta\Big(\frac{\pi_I}{\pi_J}\Big)+o(N^{-2})\nonumber
	\end{align}
	where $D_{KL}$ is the Kullback--Leibler divergence and $\pi_J$ is the Jeffreys prior which is the volume form of the statistical manifold. Each risk function indicates how far a given Bayesian predictive density is from the true distribution in the Kullback--Leibler divergence in average. Sine better priors are obtained from smaller risk functions, the priors outperforming the Jeffreys prior make the above expression greater than zero. Since the first term on the right-hand side is non-negative, the risk function of the Komaki prior is decreased with respect to the risk function of the Jeffreys prior if a prior function $\psi=\pi_I/\pi_J$ is superharmonic. If a superharmonic prior function $\psi$ can be found, it is possible to do better Bayesian prediction in the viewpoint of information theory. In the same paper, Komaki also pointed out that shrinkage priors are information-theoretically more improved in prediction than the Jeffreys prior if and only if the square root of a prior function is superharmonic.

	Since Komaki's paper \cite{Komaki:2006}, several superharmonic priors for the AR models have been found \cite{Tanaka:2006, Tanaka:2009, Choi:2014c}. The Komaki prior for the AR(2) model in the pole coordinates \cite{Tanaka:2006} is given by 
	\begin{equation}
		\psi=1-\xi^1\xi^2\nonumber
	\end{equation}
	where $\xi^i$ is a pole of the transfer function. Tanaka \cite{Tanaka:2009} generalized the two-dimensional case to superharmonic priors for the AR model in an arbitrary dimension $p$. The shrinkage prior function for the AR($p$) model is in the form of
	\begin{equation}
		\psi=\prod_{i<j}^p(1-\xi^i\xi^j)\nonumber
	\end{equation}
	where $\xi^i$ is a pole of the AR transfer function. 
	
	As mentioned before, one of the advantages in the K\"ahlerian description is that finding the Komaki prior functions becomes more efficient than those in non-K\"ahler description because the Laplace--Beltrami operators on K\"ahler manifolds are in the simpler forms. For a differentiable function~$\psi$, the Laplace--Beltrami operator in the K\"ahler geometry is represented with
	\begin{equation}
		\Delta \psi=2g^{i\bar{j}}\partial_i\partial_{\bar{j}} \psi.\nonumber
	\end{equation}
	Meanwhile, the Laplace--Beltrami operator on a non-K\"ahler manifold is expressed as
	\begin{align}
		\Delta \psi&=\frac{1}{\sqrt{\mathcal{G}}}\partial_i\Big(\sqrt{\mathcal{G}}g^{ij}\partial_j\psi\Big)\nonumber\\
		&=g^{ij}\partial_i\partial_j\psi+\frac{1}{2}g^{ij}\partial_i \log{\mathcal{G}}\partial_j \psi+\partial_i g^{ij}\partial_{j}\psi\nonumber
	\end{align}
	where $\mathcal{G}$ is the determinant of the metric tensor. It is obvious that additional calculations for the latter two terms in the right-hand side are indispensable in the non-K\"ahler cases.
	
	With the computational benefits on the K\"ahlerian information manifolds, the superharmonic prior function for the K\"ahler-AR(2) model \cite{Choi:2014c} is found
	\begin{equation}
		\psi=(1-|\xi^1|^2)(1-\xi^1\bar{\xi}^2)(1-\xi^2\bar{\xi}^1)(1-|\xi^2|^2)\nonumber
	\end{equation}
	where $\xi^i$ is the $i$-th pole of the transfer function and $\bar{\xi}^i$ is the complex conjugate of $\xi^i$. However, its generalization to any arbitrary dimensions has been unknown. Moreover, the Komaki priors for the ARMA models and the ARFIMA models are not reported yet.
	
\section{Geometric Shrinkage Priors}
\label{sec_geo_prior_algo_prior}

	As shown in the previous section, K\"ahler manifolds in information geometry are useful in order to obtain the superharmonic priors. In this section, we introduce an algorithm to find the geometric shrinkage priors by using the properties of K\"ahler geometry. Moreover, several ans\"atze for the priors are suggested.
	
	For further discussions, let us set $\tau=u^*-\kappa(\boldsymbol{\xi}, \bar{\boldsymbol{\xi}})$ where $u^*$ is a constant in $\boldsymbol{\xi}=(\xi^1,\xi^2,\cdots,\xi^n)$ and its complex conjugate $ \bar{\boldsymbol{\xi}}$. The following lemma is worthwhile when the algorithm for the prior functions is constructed.
\begin{lem}
	\label{lemma_super_sub}
	On a K\"ahler manifold, a function $\psi(\boldsymbol{\xi},\bar{\boldsymbol{\xi}})$ is superharmonic if $\psi(\boldsymbol{\xi},\bar{\boldsymbol{\xi}})$ is in the form of $\psi(\boldsymbol{\xi},\bar{\boldsymbol{\xi}})=\Psi(u^*-\kappa(\boldsymbol{\xi},\bar{\boldsymbol{\xi}}))$ such that $\kappa$ is subharmonic (or harmonic) and $\Psi'(\tau) >0, \Psi''(\tau)\le0$ (or $\Psi'(\tau) >0, \Psi''(\tau)<0$).
\end{lem}
\begin{proof}
	The Laplace--Beltrami operator on $\psi$ is given by
	\begin{align}
		\Delta \psi&=2g^{i\bar{j}}\partial_i \partial_{\bar{j}}\psi=2g^{i\bar{j}}\partial_i\Big(\big(-\partial_{\bar{j}}\kappa\big)\Psi'\Big)\nonumber\\
		&=2\Psi'' g^{i\bar{j}}\partial_i\kappa\partial_{\bar{j}}\kappa-2\Psi'g^{i\bar{j}}\partial_i\partial_{\bar{j}}\kappa\nonumber\\
		&=2\Psi''||\partial \kappa||^2_{g}-\Psi'\Delta\kappa\nonumber
	\end{align}
	where the derivatives on $\Psi$ are taken with respect to $\tau$. It is obvious that if $\kappa$ is subharmonic (or harmonic) and if $\Psi'(\tau)>0, \Psi''(\tau)\le0$ (or $\Psi'(\tau)>0, \Psi''(\tau)<0$), then the right-hand side is negative, {\em i.e.}, $\psi$ is a superharmonic function.
\end{proof}

	According to Lemma \ref{lemma_super_sub}, superharmonic functions are easily obtained from subharmonic or harmonic functions by simply plugging the (sub-)harmonic functions as $\kappa$ into Lemma \ref{lemma_super_sub}.
	
	By considering that a prior function should be positive, it is able to utilize Lemma \ref{lemma_super_sub} for obtaining the superharmonic prior functions. Let us confine the function $\psi$ in Lemma \ref{lemma_super_sub} to be positive.
\begin{thm}
	\label{thm_super_sub}
	On a K\"ahler manifold, a positive function $\psi=\Psi(u^*-\kappa)$ is a superharmonic  prior function if $\kappa$ is subharmonic (or harmonic) and $\Psi'(\tau)>0$, $\Psi''(\tau)\le0$ (or $\Psi'(\tau) >0, \Psi''(\tau)<0$).
\end{thm}
\begin{proof}
	Since this is a special case of Lemma \ref{lemma_super_sub}, the proof is obvious.
\end{proof}

	Although any (sub-)harmonic function $\kappa$ can be used for constructing superharmonic priors, restriction on $\kappa$ makes finding the ans\"atze of the geometric priors easier. From now on, upper-bounded functions are only our concerns. Additionally, we assume that $\kappa$ and $u^*$ are real. With these assumptions, it is possible to set $u^*$ as a constant greater than the upper bound of $\kappa$ in order for $\tau$ to be positive.
	
	 Ans\"atze for $\Psi$ can be found in the following example.
\begin{exmp}

	Given subharmonic (or harmonic) $\kappa$ and positive $\tau$, {\em i.e.}, upper-bounded $\kappa$, the following functions are candidates for $\Psi$
	\begin{align}
		\Psi_1(\tau)&=\tau^a\nonumber\\
		\Psi_2(\tau)&=\log{(1+\tau^a)}\nonumber
	\end{align}
	where $0<a\le1$ (or $0<a<1$).
\end{exmp}
\begin{proof}

	We only cover a subharmonic case for $\kappa$ here and it is also straightforward for the harmonic case. First of all, $\Psi_1$ and $\Psi_2$ are all positive. For $\Psi_1$, it is easy to verify the followings:
	\begin{align}
		\Psi_1'(\tau)&=a\tau^{a-1}>0\nonumber\\
		 \Psi_1''(\tau)&=a(a-1)\tau^{a-2}\le0\nonumber
	\end{align}
	for $0<a\le1$. The similar calculation is repeated for $\Psi_2$:
	\begin{align}
		\Psi_2'(\tau)&=\frac{a\tau^{a-1}}{(1+\tau^{a})}>0\nonumber\\
		\Psi_2''(\tau)&=\frac{a\tau^{a-2}(a-(1+\tau^a))}{(1+\tau^a)^2}\le0\nonumber
	\end{align}
	for $0<a\le1$. 
	
	Both functions $\Psi_1$ and $\Psi_2$ satisfy the conditions for $\Psi$ in Lemma \ref{lemma_super_sub}.
\end{proof}

	It is also possible to find ans\"atze for upper-bounded subharmonic $\kappa$. The following functions are candidates for upper-bounded and subharmonic $\kappa$.
\begin{exmp}
	For positive real numbers $a_r$ and $b_i$, the following subharmonic functions are candidates for $\kappa$ in the cases that those are upper-bounded:
	\begin{align}
		\kappa_1&=\mathcal{K}\nonumber\\
		\kappa_2&=\sum_{r=0}^{\infty} a_r |h_r(\boldsymbol{\xi})|^2\nonumber\\
		\kappa_3&=\sum_{i=1}^{n} b_i |\xi^i|^2.\nonumber
	\end{align}
\end{exmp}
\begin{proof}
	Let us assume that the ans\"atze are upper-bounded in given domains. For $\kappa_1$, it is easy to show that the K\"ahler potential $\mathcal{K}$ is subharmonic:
	\begin{align}
		\Delta \kappa_1&=\Delta \mathcal{K}=2g^{i\bar{j}}\partial_i\partial_{\bar{j}}\mathcal{K}\nonumber\\
		&=2g^{i\bar{j}}g_{i\bar{j}}=2n>0.\nonumber
	\end{align}
	
	The proof for subharmonicity of $\kappa_2$ is as follows:
	\begin{align}
		\Delta \kappa_2&=\Delta \bigg(\sum_{r=0}^{\infty} a_r |h_r(\boldsymbol{\xi})|^2\bigg)=2g^{i\bar{j}}\partial_i\partial_{\bar{j}}\bigg(\sum_{r=0}^{\infty} a_r |h_r(\boldsymbol{\xi})|^2\bigg)\nonumber\\
		&=\sum_{r=0}^{\infty} 2 a_r g^{i\bar{j}}\partial_i h_r \partial_{\bar{j}} \bar{h}_r=\sum_{r=0}^{\infty} 2 a_r ||\partial h_r ||^2_{g}>0.\nonumber
	\end{align}
	
	The subharmonicity of $\kappa_3$ is tested by
	\begin{equation}
		\Delta \kappa_3=\Delta \bigg(\sum_{i=1}^{n} b_i |\xi^i|^2\bigg)=2g^{i\bar{j}}\partial_i\partial_{\bar{j}}\bigg(\sum_{i=1}^{n} b_i |\xi^i|^2\bigg)=\sum_{i=1}^{n} 2 b_i g^{i\bar{i}}>0.\nonumber
	\end{equation}
	
	If the upper-boundedness is satisfied, the above subharmonic functions are ans\"atze for $\kappa$.\end{proof}
	
	Superharmonic prior functions on the K\"ahler manifolds are efficiently constructed from the following algorithm which exploits Theorem \ref{thm_super_sub} and the ans\"atze for $\Psi$ and $\kappa$. When we find positive and superharmonic functions, it is automatically the Komaki-style prior functions as usual. If positive, upper-bounded, and (sub-)harmonic functions are found, those functions are plugged into Theorem \ref{thm_super_sub} in order to obtain superharmonic prior functions. Multiplying the Jeffreys prior by the superharmonic prior functions, we finally acquire the geometric shrinkage priors. Additionally, since the ans\"atze are already given, there is no extra cost to find the Komaki prior functions except for verifying whether or not the information geometry is a K\"ahler manifold. Comparing with the literature on the Komaki priors of the time series models \cite{Tanaka:2006, Tanaka:2009, Choi:2014c}, obtaining the geometric priors on the K\"ahler manifolds becomes more efficient and more robust.
	
\section{Example: ARFIMA Models}

\label{sec_geo_prior_example}

	The ARFIMA model is the generalization of the ARMA model with a fractional differencing parameter in order to model the long memory process. The transfer function of the ARFIMA$(p,d,q)$ model with parameters $\boldsymbol{\xi}=(\xi^{-1},\xi^{0},\xi^1,\cdots,\xi^{p+q})=(\sigma,d,\lambda_1,\cdots,\lambda_p,\mu_1,\cdots,\mu_q)$ is given by
	\begin{equation*}
		h(z;\boldsymbol{\xi})=\frac{\sigma^2}{2\pi}\frac{(1-\mu_1 z^{-1})(1-\mu_2 z^{-1})\cdots(1-\mu_q z^{-1})}{(1-\lambda_1 z^{-1})(1-\lambda_2 z^{-1})\cdots(1-\lambda_p z^{-1})}(1-z^{-1})^d
	\end{equation*}
where $d$ is the differencing parameter and $\mu_i, \lambda_i, \sigma$ are a pole, a root, and a gain in the ARMA model, respectively. It is noteworthy that the transfer function of the ARFIMA model is decomposed into the ARMA model part and the fractionally integration part. Additionally, every poles and roots of the linear system are located inside the unit disk, {\em i.e.}, $|\lambda_i|<1$ for $i=1,\cdots,p$ and $|\mu_i|<1$ for $i=1,\cdots,q$. 
	
	Similar to the ARMA case \cite{Choi:2014c}, the full geometry of the ARFIMA model is a K\"ahler manifold and the submanifold of a constant gain $\sigma$ is also K\"ahler geometry. This submanifold also exhibits the explicit Hermitian condition on the metric tensor. It is easy to cross-check the Hermitian structure by fixing $h_0=1$ up to the gain of the signal filter. We will work on this submanifold.
	
	Since the information geometry of the ARFIMA model is a K\"ahler manifold, the K\"ahler potential of the ARFIMA geometry  is obtained from the square of the Hardy norm of the logarithmic transfer function (or the square of the complex cepstrum norm), Equation  (\ref{eqn_kahler_potential_formula}), represented with
	\begin{equation}
		\label{eqn_kahler_potential_ARFIMA}
		\mathcal{K}=\sum_{r=1}^{\infty}\Big|\frac{d+(\mu_1^r+\cdots+\mu_q^r)-(\lambda_1^r+\cdots+\lambda_p^r)}{r}\Big|^2.
	\end{equation}
	It is obvious that the K\"ahler potential for the ARFIMA model, Equation  (\ref{eqn_kahler_potential_ARFIMA}), is reducible to the K\"ahler potential of the ARMA geometry by setting $d=0$. It is easy to verify that the K\"ahler potential of the ARFIMA geometry is upper-bounded by $(d+p+q)^2\frac{\pi^2}{6}$.
	
	By using Equation  (\ref{eqn_metric_kahler_formula}), the metric tensor of the K\"ahler geometry is simply derived from the K\"ahler potential. The metric tensor of the K\"ahler-ARFIMA geometry is given by
	\begin{displaymath} 
		g_{i\bar{j}} = \left( \begin{array}{ccc} 
			\frac{\pi^2}{6} & \frac{1}{\bar{\lambda}_j} \log{(1-\bar{\lambda}_j)} & -\frac{1}{\bar{\mu}_j} \log{(1-\bar{\mu}_j)} \\  \frac{1}{\lambda_i} \log{(1-\lambda_i)} & \frac{1}{1-\lambda_i\bar{\lambda}_j} & -\frac{1}{1-\lambda_i\bar{\mu}_j} \\  -\frac{1}{\mu_i}\log{(1-\mu_i)}  & -\frac{1}{1-\mu_i\bar{\lambda}_j} & \frac{1}{1-\mu_i\bar{\mu}_j} \end{array} \right)
	\end{displaymath}
	and it is easy to show that the metric tensor contains the pure ARMA metric. The metric tensor is also in the similar form to the ARFIMA geometry in non-complexified coordinates \cite{Ravishanker:2001p5836}. The metric tensor indicates that the ARMA geometry is embedded in the ARFIMA geometry and corresponds to the submanifold of the ARFIMA manifold. The ARMA part of the metric tensor is the same metric with the K\"ahler-ARMA geometry in Choi and Mullhaupt \cite{Choi:2014c}. In addition to that, we can cross-check the fact that the ARMA geometry is also a K\"ahler manifold based on a property of a K\"ahler manifold that a submanifold of the K\"ahler geometry is K\"ahler.
	
	Other geometric objects can be derived from the metric tensor. For example, the non-trivial components of the 0-connection are given by Equation  (\ref{eqn_connection_formula}). It is noteworthy that any connection components with the $d$-coordinate in the first two indices of the connection are trivially zero and the others might not be vanishing. Similar to the 0-connection, the Ricci tensor components along the fractionally integrated direction are also zero because there is no dependence on $d$ in the metric tensor. Considering the Schur complement, the non-vanishing Ricci tensor components are decomposed into the Ricci tensor from the pure ARMA part and the term from the mixing between the ARMA part and the fractionally integrated (FI) part:
	\begin{equation}
		R_{i\bar{j}}=R_{i\bar{j}}^{ARMA}+R_{i\bar{j}}^{ARMA-FI}\nonumber
	\end{equation} 
	where $i$ and $j$ are not along the $d$-coordinate. 
	
	It is the time to be back to the geometric shrinkage priors. Since the K\"ahler potential of a given ARFIMA model is upper-bounded by a constant $u^*=(d+p+q)^2\frac{\pi^2}{6}$, the intrinsic priors on the K\"ahler manifold can be found as it is proven in the previous section. By using the algorithm and the ans\"atze related to the K\"ahler potential, some geometric shrinkage prior functions for the ARFIMA model are constructed as
	\begin{align}
		\psi_1&=(u^*-\mathcal{K})^a\nonumber\\
		\psi_2&=\log{(1+(u^*-\mathcal{K})^a)}\nonumber
	\end{align}
	where $0<a\le1$. It is also noteworthy that when $d=0$ in the K\"ahler potential, superharmonic priors of the ARMA (or AR/MA) models are obtained and finding the priors becomes much simpler than the literature on the Komaki priors of the AR models \cite{Tanaka:2006, Tanaka:2009, Choi:2014c}. Similarly, $\kappa_2$ and $\kappa_3$ are also utilized for the superharmonic prior function ans\"atze in the ARFIMA models because the both functions are upper-bounded on the ARFIMA manifold. Moreover, if we set $d=0$ for $\kappa_2$ or $b_0=0$ for $\kappa_3$, the ans\"atze for the ARFIMA models are reducible to the Komaki priors of the ARMA models.
	
\section{Conclusion}
\label{sec_geo_prior_conclusion}

	In this paper, we build up an algorithm and ans\"atze for the geometric shrinkage priors of K\"ahlerian signal filters. By using the properties of K\"ahler manifolds, an algorithm to find the Komaki priors is constructed and ans\"atze for the prior functions are suggested. Additionally, some ans\"atze associated with the K\"ahler potential are geometrically intrinsic to K\"ahlerian information manifolds because the geometry is derived from the K\"ahler potential which is the square of the complex cepstrum norm of a linear system.
	
	Comparing with the literature on the Komaki priors of the time series models, verification of the geometric priors is much easier on the K\"ahler manifold and it is also possible to acquire the geometric shrinkage priors for highly complicated models in the more efficient and robust way. For example, Bayesian predictive priors for the ARFIMA model are obtained from the algorithm and  ans\"atze for the prior functions. The shrinkage priors of the ARMA cases are simply found from the geometric shrinkage priors of the ARFIMA models by using the property of submanifolds in the K\"ahler geometry.

\section*{Acknowledgements}
	We are thankful to Michael Tiano for useful discussions.

\end{document}